\documentclass{aims}

\usepackage{amsmath}

  \usepackage{paralist}
  \usepackage{graphics} 
  \usepackage{epsfig} 
\usepackage{graphicx}  \usepackage{epstopdf}
 \usepackage[colorlinks=true]{hyperref}
\hypersetup{urlcolor=blue, citecolor=red}

\def\currentvolume{X}
 \def\currentissue{X}
  \def\currentyear{200X}
   \def\currentmonth{XX}
    \def\ppages{X--XX}
     \def\DOI{10.3934/xx.xx.xx.xx}
     
\newtheorem{theorem}{Theorem}[section]

\newtheorem{main}{Main Theorem}   
\newtheorem{lemma}[theorem]{Lemma}

\newtheorem*{problem}{Problem}
\theoremstyle{definition}
\newtheorem{definition}[theorem]{Definition}
\newtheorem{assumption}[theorem]{Assumption} 
\newtheorem{remark}{Remark}
\newtheorem{example}[theorem]{Example}  

\def\clo#1{\overline{#1}}
\def\text#1{\mbox{#1}}

\newcommand{\ra}{\rangle}
\newcommand{\la}{\langle}
\newcommand{\Sop}{\mathcal{S}}
\newcommand{\Cop}{\mathcal{C}}
\newcommand{\Pop}{\mathcal{P}}
\newcommand{\Aop}{\mathcal{A}}
\newcommand{\Top}{\partial_{t}^{-1}}
\newcommand{\Hsp}{\mathcal{H}}

\usepackage{amssymb}

\title[A control approach to recover the wave speed]
      {A control approach to recover the wave speed (conformal factor) from one measurement}

\author[Sebastian Acosta]{}



 \email{sacosta@bcm.edu}

\thanks{The author was partially supported by AFOSR Grant
FA9550-12-1-0117 and ONR Grant N00014-12-1-0256}

\begin{document}
\maketitle

\centerline{\scshape Sebastian Acosta}

\medskip

{\footnotesize
 \centerline{Department of Pediatric Cardiology}
   \centerline{Baylor College of Medicine, Houston, TX, USA}
}

\bigskip

 \centerline{(Communicated by the associate editor name)}

\begin{abstract}
In this paper we consider the problem of recovering the conformal factor in a conformal class of Riemannian metrics from the boundary measurement of one wave field. More precisely, using boundary control operators, we derive an explicit equation satisfied by the contrast between two conformal factors (or wave speeds). This equation is Fredholm and generically invertible provided that the domain of interest is properly illuminated at an initial time. We also show locally Lipschitz stability estimates.
\end{abstract}

\section{Introduction} \label{Section:Intro}
In some applications, such as imaging of soft biological tissues, the propagation of pressure waves can be modeled by the acoustic wave equation. Then one seeks to reconstruct the wave speed within a region of interest. However, in some cases, it is important to account for the anisotropic properties of media (muscles and bones). Although these are more appropriately modeled by elastodynamic systems, the anisotropic scalar wave equation may serve as a first step to incorporate such behavior. This equation can be expressed in geometric terms as follows,
\begin{align}
\partial_{t}^{2}w = \mathcal{A}_{c^{-2} g} w, \label{Eqn.2}
\end{align}
where $\mathcal{A}_{g}$ is the $\mu$-weighted Laplace-Beltrami operator defined as
\begin{align}
\mathcal{A}_{g} w = \mu^{-1} \text{div}_{g} \, ( \mu \, \text{grad}_{g} w) . \label{Eqn.3}
\end{align}
Here $\mu(x)$ and $c(x)$ are sufficiently smooth and positive, and $\text{div}_{g}$ and $\text{grad}_{g}$ denote the divergence and the gradient with respect to a smooth Riemannian metric $g$. In this paper, we assume full knowledge of the coefficient $\mu$ and the metric $g$.

The inverse problem we consider is to recover the conformal factor in the conformal class of metrics represented by $g$, that is, we seek to recover the wave speed $c(x)$ from knowledge of boundary data of $w$ on a sufficiently large time interval. Abusing terminology, we shall refer to $c(x)$ as the \textit{wave speed} even though the actual wave speed could be anisotropic and would be fully characterized by both $c$ and $g$.

In the literature, most works concerned with the recovery of coefficients in the wave equation from boundary measurements are limited to one of the following cases:
\begin{itemize}
\item[(i)] Knowledge of the hyperbolic Dirichlet-to-Neumann map. Thus, it is usually assumed that infinitely many illuminations and corresponding measurements are known. Here we find approaches based on special solutions and integral geometry (see for instance \cite{Rak-Sym-1988,Sun-1990,Syl-Uhl-1991,Ste-Uhl-1998,Uhl-2000,Ste-Uhl-2005,Bel-Dos-2011,Liu-Oks-2014} and  \cite[Section 8.3]{Isakov-1998}), and the boundary control method (see \cite{Beli-1997,Beli-2007} and \cite[Section 8.4]{Isakov-1998}).  

\item[(ii)] Knowledge of initial data and the corresponding boundary measurement. These are results in the general form of Theorem 8.2.2 in \cite{Isakov-1998} originated from \cite{Buk-Kli-1981}. See also \cite{Kli-1992,Puel-Yam-1996,Puel-Yam-1997,Yam-1999,Ima-Yam-2001,Ima-Yam-2003,Kli-Yam-2006,Bel-Yam-2006,Bel-Yam-2008,Ste-Uhl-2013,Liu-Tri-2011,Liu-Tri-2012,Liu-Tri-2013,Liu-2013} and references therein.
\end{itemize}

This paper falls into the second category. We summarize the novelty and contributions of our work in the following points:
\begin{itemize}
\item[(a)] The wave speed $c(x)$ appears within the principal part of the hyperbolic equation which leads to a challenging interplay between $c(x)$ and its derivatives. So for instance, our approach can be employed to determine a scalar coefficient $p(x)$ in $\nabla \cdot p \nabla$ treated in \cite{Ima-Yam-2003,Kli-Yam-2006}. Our work can be seen as a generalization of \cite{Ima-Yam-2003,Kli-Yam-2006} and other cited references to non-Euclidean geometries.
\item[(b)] Using control operators, we derive an \textit{explicit} equation satisfied by the contrast between two wave speeds. This equation has Fredholm form provided that the domain of interest is properly illuminated at an initial time. See the main results in Section \ref{Section:Inverse} for the precise statements. In dimension $n=2$, this Fredholm equation has a principal part of order zero (see Theorem \ref{Thm.MainInv}). In dimension $n \geq 3$, we can derive an analogous equation using finitely many illuminations (see Theorem \ref{Thm.MainInv2}). Otherwise, with a single illumination, we obtain a first-order PDE. We give conditions on the initial illumination for this first-order equation to be Fredholm (see Theorem \ref{Thm.MainInv3}). 
\item[(c)] The main restriction is that the initial profile $w|_{t=0}$ is assumed to be known a-priori. However, in contrast to most previous publications, we make no assumptions about the initial \textit{velocity} profile $\partial_{t} w |_{t=0}$.
\item[(d)] If well-known geometrical conditions are satisfied, this control approach is naturally suited for measurements on a sub-boundary $\Gamma \subset \partial \Omega$. See Assumption \ref{Assump.001} below.
\end{itemize}

We wish to acknowledge that our work is inspired by a combination of ideas developed in Puel--Yamamoto \cite{Puel-Yam-1997}, Liu--Oksanen \cite{Liu-Oks-2014} and Stefanov--Uhlmann \cite{Ste-Uhl-2013}. A brief review of control theoretical tools is presented in Section \ref{Section:Control}. Our main results are stated in Section \ref{Section:Inverse} and the proofs are provided in Section \ref{Section:Proof}.

\section{Background on Control Theory} \label{Section:Control}

Our approach relies heavily on exact boundary controllability for the wave equation. Hence, the purpose of this section is to review some facts and define the notation concerning control theoretical tools. Our guiding references are \cite{Glow-Lions-He-2008,Bar-Leb-Rau-1992,GLLT-2004,Lions-1988}. The first item is to ensure that the media under consideration allows for exact controllability which is closely related to geometric notions \cite{Glow-Lions-He-2008,Bar-Leb-Rau-1992,GLLT-2004}. We consider the wave equation (\ref{Eqn.2}) defined in a sufficiently smooth simply connected domain $\Omega \subset \mathbb{R}^n$ with boundary $\partial \Omega$ where $n \geq 2$, and the geodesic flow associated with this wave equation is determined by the metric $c^{-2} g$ because,
\begin{align*}
& \mathcal{A}_{c^{-2}g} w = \Delta_{c^{-2} g} w \, + \, \text{lower order terms},
\end{align*}
where $\Delta_{c^{-2} g}$ is the Laplace-Beltrami operator corresponding to the metric $c^{-2}g$.

Following Bardos, Lebeau and Rauch \cite{Bar-Leb-Rau-1992}, we assume that our problem enjoys the \textit{geometrical control condition} (GCC) for the Riemannian manifold $(\Omega, c^{-2} g)$ with only a portion $\Gamma$ of the boundary $\partial \Omega$ being accessible for control and observation. We assume that $\Gamma$ is a smooth and simply connected domain relative to $\partial \Omega$. In this paper, we assume that the following condition is satisfied.

\begin{assumption}[GCC] \label{Assump.001}
The geodesic flow of $(\Omega, c^{-2} g)$ reaches the accessible part of the boundary. In other words, there exists $\tau < \infty$ such that any unit-speed ray, originating from any point in $\Omega$ at $t=0$, reaches $\Gamma$ in a nondiffractive manner (after possible geometrical reflections on $\partial \Omega \setminus \Gamma$) before time $t=\tau$.
\end{assumption}

To fix our ideas, we will pose the control problem with a specific level of regularity. However, we should keep in mind that this problem can also be formulated on any scale of Sobolev regularity \cite{Bar-Leb-Rau-1992}. In addition, we should interpret the Hilbert space $H^{0}(\Omega)$ with the inner-product appropriately weighted by $\mu \det(c^{-2} g)^{1/2}$ so that $\mathcal{A}_{c^{-2} g}$ is formally self-adjoint with respect to the duality pairing of $H^{0}(\Omega)$. The same weight is incorporated in the inner product for $H^{0}(\partial \Omega)$.

Now we consider the following auxiliary problem. Given $\zeta \in H^{1}_{0}((0,\tau) \times \Gamma)$, find the generalized solution $\xi \in C^{k}([0,\tau];H^{1-k}(\Omega))$ for $k=0,1$ of the following time-reversed problem
\begin{subequations}
\label{Eqn:c}
\begin{align}
& \partial_{t}^{2} \xi - \mathcal{A}_{c^{-2} g} \xi = 0 \qquad && \mbox{in} \quad (0,\tau) \times \Omega \label{Eqn:001c}  \\
& \xi = 0 \quad \mbox{and} \quad \partial_{t} \xi = 0 \qquad && \mbox{on} \quad \{ t = \tau \} \times \Omega \label{Eqn:002c} \\
& \xi = \zeta \qquad && \mbox{on} \quad (0,\tau) \times \Gamma \label{Eqn:003c} \\
& \xi = 0 \qquad && \mbox{on} \quad (0,\tau) \times (\partial \Omega \setminus \Gamma) \label{Eqn:004c} 
\end{align}
\end{subequations}
This is a well-posed problem satisfying a stability estimate of the following form,
\begin{align*}
\| \xi \|_{C^{k}([0,\tau];H^{1-k}(\Omega))} + \| \nabla_{c^{-2} g} \xi \|_{C^{k}([0,\tau];H^{-k}(\Omega))} \lesssim \| \zeta \|_{H^{1}_{0}((0,\tau) \times \Gamma)}. 
\end{align*}
for $k=0,1$. Here and in the rest of the paper, the symbol $\lesssim$ means inequality up to a positive constant.

Given arbitrary $\phi \in H^{0}(\Omega)$, the goal of the control problem is to find a Dirichlet boundary condition $\zeta \in H^{1}_{0}((0,\tau) \times \Gamma)$ to drive the solution $\xi$ of (\ref{Eqn:c}) from vanishing Cauchy data at $t=\tau$ to the desired Cauchy data $( \xi, \partial_{t} \xi) = (0,\phi)$ at time $t=0$. 

The well-posedness of this control problem under the \textit{geometrical control condition} was obtained in \cite{Bar-Leb-Rau-1992} for smooth data, and extended in \cite{Burq-1997} for less regular domains and coefficients. For our reference, we state it now as a theorem.

\begin{theorem}[Controllability] \label{Thm.Control}
Let the geometrical control condition \ref{Assump.001} hold. Then for any function $\phi \in H^{0}(\Omega)$, there exists Dirichlet boundary control $\zeta \in H^{1}_{0}((0,\tau) \times \Gamma)$ so that the solution $\xi$ of (\ref{Eqn:c}) satisfies
\begin{align*}
(\xi , \partial_{t} \xi) = (0,\phi) \qquad \text{at time $t = 0$.}
\end{align*}

Among all such boundary controls there exists $\zeta_{\rm min}$ which is uniquely determined by $\phi$ as the minimum norm control and satisfies the following stability condition
\begin{align*}
& \| \zeta_{\rm min} \|_{H^{1}_{0}((0,\tau) \times \Gamma)} \leq C \| \phi \|_{H^{0}(\Omega)}
\end{align*}
for some positive constant $C = C(\Omega,\Gamma,c,g,\tau)$.
\end{theorem}

Notice that we drive the wave field from vanishing Cauchy data at time $t=\tau$ to Cauchy data at time $t=0$ of the form $(\xi,\partial_{t}\xi) = (0,\phi)$. Now we define a particular way to drive the solution to a Cauchy data at time $t=0$ of the form $(\xi,\partial_{t}\xi) = (\phi,0)$.

\begin{definition}[Control Operators] \label{Def:CSOpe}
Let the \textit{control} operator
\begin{align*}
\Cop : H^{0}(\Omega) \to H^{0}((0,\tau) \times \Gamma), 
\end{align*}
be given by the map $\phi \mapsto \partial_{t} \zeta_{\rm min}$ where $\zeta_{\rm min}$ is defined in Theorem \ref{Thm.Control}. We also define the corresponding \textit{solution} operator
\begin{align*}
\Sop : H^{0}(\Omega) \to H^{0}((0,\tau) \times \Omega), 
\end{align*}
mapping $\phi \mapsto \partial_{t} \xi$ where $\xi$ is the solution of (\ref{Eqn:c}) with $\zeta = \zeta_{\rm min}$.
\end{definition}

The \textit{control} and \textit{solution} operators given by Definition \ref{Def:CSOpe} are purposely defined as the time derivative of more regular functions so that they enjoy a gain of space regularity with integration in time. Denote by $\Top$ the following operation
\begin{align*}
(\Top v)(t) = \int_{0}^{t} v(s) ds.  
\end{align*}
Using the above notation, we see that $\Top \partial_{t} \xi = \xi$ because $\xi|_{t=0} = 0$, and $\Top \partial_{t} \zeta = \zeta$ because $\zeta |_{t=0} = 0$ since $\zeta \in H^{1}_{0}((0,\tau) \times \Gamma)$. This leads to the following properties for the \textit{control} and \textit{solution} operators in the form of a lemma, whose proof follows directly from the conclusions in Theorem \ref{Thm.Control} and well-known results \cite{Lio-Mag-Book-1972,Lasiecka-1986}.

\begin{lemma} \label{Lemma.007}
The \textit{control} and \textit{solution} operators from Definition \ref{Def:CSOpe} are bounded. 

Moreover, for given $\phi \in H^{0}(\Omega)$ we have that Dirichet control $\eta = \Cop \phi$ drives the generalized solution $\psi = \Sop \phi$ of (\ref{Eqn:c}) from vanishing Cauchy data at time $t=\tau$ to Cauchy data $(\psi,\partial_{t} \psi) = (\phi,0)$ at time $t=0$. Also, the wave field $\psi = \Sop \phi$ satisfies the following estimate,
\begin{align}
\| \psi \|_{C^{k}([0,\tau];H^{-k}(\Omega))} + \| \nabla_{c^{-2} g} \psi \|_{C([0,\tau];H^{-k-1}(\Omega))} \lesssim \| \eta \|_{H^{0}((0,\tau) \times \Gamma)},  \label{Eqn:Stab02}
\end{align}
for $k=0,1$, and 
\begin{align}
\| \Top \psi \|_{C^{k}([0,\tau];H^{1-k}(\Omega))} + \| \nabla_{c^{-2} g} \Top \psi \|_{C^{k}([0,\tau];H^{-k}(\Omega))} \lesssim \| \Top \eta \|_{H^{1}_{0}((0,\tau) \times \Gamma)}  \label{Eqn:Stab03}
\end{align}
for $k=0,1$.
\end{lemma}

\section{Inverse Problem and Recovery Equation} \label{Section:Inverse}

In this section we state the inverse problem for the recovery of the wave speed, and we also present our main results. We start by defining the forward problem. Let $w$ be the strong solution of the following problem,
\begin{subequations}
\label{Eqn:Forward}
\begin{align}
& \partial_{t}^{2} w - \mathcal{A}_{c^{-2} g} w = 0 \qquad  && \text{in} \quad (0 ,\tau) \times \Omega,   \\
& w = \alpha \quad \text{and} \quad \partial_{t} w = \beta && \text{on} \quad \{ t = 0 \} \times \Omega,  \\
& w = \gamma \qquad  && \text{on} \quad (0 ,\tau) \times \Gamma, \\
& w = 0 \qquad  && \text{on} \quad (0 ,\tau) \times (\partial \Omega \setminus \Gamma). 
\end{align}
\end{subequations}
We also define $\tilde{w}$ to solve a wave problem with potentially different speed $\tilde{c}$, and initial Cauchy data $(\tilde{w},\partial_{t}\tilde{w})|_{t=0}=(\alpha, \tilde{\beta})$ and the same Dirichlet boundary values $\tilde{w} = \gamma$ on $(0 ,\tau) \times \Gamma$. Notice that $\beta$ and $\tilde{\beta}$ do \textit{not} have to coincide. We define the Neumann traces $\lambda$ and $\tilde{\lambda}$ of $w$ and $\tilde{w}$, respectively, as follows
\begin{align*}
\lambda  = \nu \cdot \text{grad}_{c^{-2} g} w \quad \text{and} \quad  \tilde{\lambda}  = \nu \cdot \text{grad}_{c^{-2} g} \tilde{w} \qquad \text{on $((0,\tau) \times \Gamma)$},
\end{align*}
where $\nu$ denotes the outward normal on $\partial \Omega$.

Now we realize that the governing operator in (\ref{Eqn.2}) can be expressed as follows,
\begin{align}
& \mathcal{A}_{c^{-2}g} w = c^{2} \mathcal{A}_{g} w + \frac{2-n}{2} \,  \nabla c^{2} \cdot \nabla_{g} w \label{Eqn.4}
\end{align}
where $\nabla$ is the Euclidean gradient. In dimension $n=2$, we have the convenient identity $\mathcal{A}_{c^{-2}g} w = c^{2} \mathcal{A}_{g} w$. However, we consider the general case $n \geq 2$. In view of the identity (\ref{Eqn.4}) and in order to obtain a linear inverse source problem, we define
\begin{subequations}
\begin{align}
& u = w - \tilde{w}, \qquad f = c^2 - \tilde{c}^2,  \label{Eqn:09} \\
& \sigma = \mathcal{A}_{g} \tilde{w}, \qquad \text{and} \qquad \Sigma = \frac{2-n}{2}  \, \nabla_{g} \tilde{w}. \label{Eqn:10}
\end{align}
\end{subequations}

Hence we have that $u$ is the strong solution of the following problem,
\begin{subequations}
\label{Eqn:f}
\begin{align}
& \partial_{t}^{2} u - \mathcal{A}_{c^{-2} g} u = \sigma f +  \Sigma \cdot \nabla f  \qquad && \mbox{in} \quad (0,\tau) \times \Omega  \label{Eqn:001f} \\
& u = 0 \quad \mbox{and} \quad \partial_{t} u = \beta - \tilde{\beta} \qquad && \mbox{on} \quad \{ t = 0 \} \times \Omega \label{Eqn:002f} \\
& u = 0 \qquad && \mbox{on} \quad (0,\tau) \times \partial \Omega. \label{Eqn:003f}
\end{align}
\end{subequations}
Recall that we measure the Neumann data 
\begin{align}
m  = \lambda - \tilde{\lambda} = \nu \cdot \text{grad}_{c^{-2} g} u, \qquad \text{on $((0,\tau) \times \Gamma)$} \label{Eqn.Measurement}
\end{align}
With this notation we can precisely state the inverse problem for the determination of $f$ which is the contrast between the two conformal factors.

\begin{problem}[Inverse Problem] \label{Def.InvProb}
For an appropriate initial illumination $\alpha$, show that if the Neumann traces $\lambda$ and $\tilde{\lambda}$ coincide, then $c=\tilde{c}$. Or by conversion into a linearized inverse source problem, given the Neumann boundary measurement $m$ in (\ref{Eqn.Measurement}) and the source factors $\sigma$ and $\Sigma$ in (\ref{Eqn:f}), determine the source term $f$. 
\end{problem}

In order to state our main results, we define the following time integral operator $\Pop : H^{0}((0,\tau) \times \Omega) \to H^{0}(\Omega)$ given by
\begin{align}
(\Pop v)(x) = \int_{0}^{\tau} v(t,x) \, dt. \label{Eqn.TimeInt}
\end{align}

Now notice that $\dot{u}$ solves a wave problem (in the standard weak sense) with forcing term $\dot{\sigma} f + \dot{\Sigma} \cdot \nabla f$ and initial Cauchy data 
\begin{align*}
(\dot{u},\partial_{t}\dot{u})|_{t=0} = (\beta - \tilde{\beta}, \sigma|_{t=0}f + \Sigma|_{t=0} \cdot \nabla f).
\end{align*}

In what follows, we will evaluate the duality pairing between the time derivative of the terms in equation (\ref{Eqn:001f}) against $\psi$, the solution of (\ref{Eqn:c}). Let $m$ be given by (\ref{Eqn.Measurement}) and consider,
\begin{align*}
\la \dot{\sigma} f + \dot{\Sigma} \cdot \nabla f , \psi \ra_{(0,\tau) \times \Omega} & = \la \partial_{t}^{2} \dot{u} - \mathcal{A}_{c^{-2}g} \dot{u} , \psi \ra_{(0,\tau) \times \Omega} \\
& = - \la \sigma_{0} f + \Sigma_{0} \cdot \nabla f, \phi \ra_{\Omega} - \la \dot{m} , \eta \ra_{(0,\tau) \times \Gamma }
\end{align*}
where $\psi = \Sop \phi$ and $\eta = \Cop \phi$, and $u$ possesses vanishing trace on $(0,\tau) \times \partial \Omega$. Here we have also used the fact that $\psi$ has zero Cauchy data at time $t=\tau$, and $\partial_{t}\psi = 0$ at time $t=0$. This last fact is very important since it eliminates the unknown term $ \dot{u}|_{t=0} = \beta - \tilde{\beta}$ from the above equation. Now, since $\phi \in H^{0}(\Omega)$ is arbitrary, then we obtain the following,
\begin{align}
\Big [\Sigma_{0} + ( \Pop \dot{\Sigma} \Sop ) \Big ]^{*} \nabla f + \Big [\sigma_{0} + ( \Pop \dot{\sigma} \Sop ) \Big ]^{*} f  & = - \Cop^{*} \dot{m}, \label{Eqn.Main03}
\end{align}
where $\sigma_{0} = \sigma |_{t=0} : H^{0}(\Omega) \to H^{0}(\Omega)$ is understood as a multiplicative operator mapping $\phi \mapsto \sigma_{0}\phi$, which makes it self-adjoint. Similarly, $\Sigma_{0} = \Sigma |_{t=0} : H^{0}(\Omega) \to H^{0}(\Omega)^{n}$ maps $\phi \mapsto \Sigma_{0} \phi$ and its adjoint $\Sigma_{0}^{*} : H^{0}(\Omega)^{n} \to H^{0}(\Omega)$ maps $\Phi \mapsto \Sigma_{0} \cdot \Phi$.

This last equation is suited for our purposes and it represents the main result of this paper along with statements of its solvability provided by the following theorems.

\begin{main} \label{Thm.MainInv}
Let the geometrical control condition \ref{Assump.001} hold. If 
\begin{itemize}
\item[(i)] there exists a constant $\delta > 0$ such that $|\sigma_{0}(x)| \geq \delta$ for a.a. $x \in \Omega$, and
\item[(ii)] $\sigma \in C^{1}([0,\tau] ; C(\clo{\Omega}))$,
\end{itemize}
then the operator 
\begin{align}
\Big [\sigma_{0} + ( \Pop \dot{\sigma} \Sop ) \Big ] : H^{0}(\Omega) \to H^{0}(\Omega)  \label{Eqn:MainOp01}
\end{align}
is Fredholm. Moreover, there exists an open and dense subset $\mathcal{U}$ of $(i) \cap (ii)$ such that for each $\sigma$ in this set, the operator (\ref{Eqn:MainOp01}) and its adjoint are boundedly invertible. 

Concerning the Inverse Problem \ref{Def.InvProb}, in dimension $n=2$ where $\Sigma \equiv 0$, for each $\sigma \in \mathcal{U}$, the inverse problem is uniquely solvable, and the following locally Lipschitz stability estimate,
\begin{align*}
\| c^2 - \tilde{c}^2 \|_{H^{0}(\Omega)} \leq C \, \| \lambda -  \tilde{\lambda} \|_{H^{1}([0,\tau]; H^{0}(\Gamma))},
\end{align*}
holds for a positive constant $C$ that remains uniformly bounded for $c$ and $\tilde{c}$ in small bounded sets of $C^{2}(\clo{\Omega})$.
\end{main}

We find appropriate to make the following remarks concerning global uniqueness in the recovery of the wave speed (conformal factor) from a single measurement.

\begin{remark} \label{Rmk:01}
In dimension $n=2$ we have that $\Sigma \equiv 0$. Thus, under the geometrical control condition \ref{Assump.001}, if the initial state $\alpha$ of the forward problem (\ref{Eqn:Forward}) satisfies $|\mathcal{A}_{g} \alpha(x)| \geq \delta > 0$, then for a sufficiently regular and generic term $\mathcal{A}_{g} \alpha$, the Cauchy data of $w$ on the sub-boundary $((0,\tau) \times \Gamma)$ uniquely identifies the wave speed $c(x)$. The positivity condition $|\mathcal{A}_{g}\alpha(x)| \geq \delta > 0$ indicates that the region of interest should be properly illuminated at the initial time $t=0$. Similar positivity conditions are also found in studies of related inverse problems. See \cite[Thm. 8.2.2]{Isakov-1998} and \cite{Kli-1992,Puel-Yam-1996,Puel-Yam-1997,Yam-1999,Ima-Yam-2001,Ima-Yam-2003,Bel-Yam-2006,Bel-Yam-2008,Ste-Uhl-2013,Liu-Tri-2011,Liu-Tri-2012,Liu-Tri-2013,Liu-2013} and references therein.
\end{remark}

\begin{remark} \label{Rmk:02}
Recall the definition of $\sigma$ in (\ref{Eqn:10}). The regularity condition $(ii)$ of Theorem \ref{Thm.MainInv} requires the initial Cauchy data $(\tilde{w},\partial_{t}\tilde{w})|_{t=0}=(\alpha, \tilde{\beta})$ and the Dirichlet boundary data $\gamma$ to be sufficiently regular. We note that if $\tilde{w} \in C^2([0,\tau] \times \clo{\Omega})$ (a classical solution of the wave equation) then this regularity requirement is satisfied.
\end{remark}

\begin{remark} \label{Rmk:03}
Also notice that, aside from regularity, Theorem \ref{Thm.MainInv} does not require any a-priori knowledge of the initial velocity $\beta$ in the forward problem (\ref{Eqn:Forward}).
\end{remark}

Now, for the case $n \geq 3$ we can obtain a similar result using more than one illumination. More precisely, consider the forward problem (\ref{Eqn:Forward}) with $n+1$ triples of input data $(\alpha_{i},\beta_{i},\gamma_{i})$ for $i=1,...,n+1$. We do not impose conditions on the initial velocities $\beta_{i}$ but we do assume knowledge of $\alpha_{i}$ and $\gamma_{i}$. Each triple $(\alpha_{i},\beta_{i},\gamma_{i})$ induces a corresponding wave field $w_{i}$ with wave speed $c$, and $\tilde{w}_{i}$ with wave speed $\tilde{c}$. As in (\ref{Eqn:09}), we define  
\begin{align}
& \sigma_{i} = \mathcal{A}_{g} \tilde{w}_{i}, \qquad \text{and} \qquad \Sigma_{i} = \frac{2-n}{2}  \, \nabla_{g} \tilde{w}_{i} \qquad \text{for $i=1,...,n+1$}. \label{Eqn:11}
\end{align}
We set up the following operator-valued matrices
\begin{eqnarray*}
A = \left[
\begin{array}{ccccc}
 \Sigma_{1,1} & \Sigma_{1,2} & \cdots & \Sigma_{1,n} & \sigma_{1} \\ 
 \Sigma_{2,1} & \Sigma_{2,2} & \cdots & \Sigma_{1,n} & \sigma_{2} \\ 
 \vdots & \vdots & \ddots & \vdots & \vdots \\ 
 \Sigma_{n+1,1} & \Sigma_{n+1,2} & \cdots & \Sigma_{n+1,n} & \sigma_{n+1}
\end{array} \right] \label{Eqn:12} \quad \text{and} \quad
K_{i,j} = (\Pop \dot{A}_{i,j} \Sop)^{*} 
\end{eqnarray*}
where $\Sigma_{i,j}$ is the $j^{\rm th}$ entry of $\Sigma_{i}$. Notice that the main equation (\ref{Eqn.Main03}) can be expressed as $\left[ M_{0} + K \right] F = M$ where $F = (\nabla f , f)$ and $M_{i} = - \Cop^{*} m_{i}$, $m_{i} = \nu \cdot \nabla_{c^{-2} g} (w_{i} - \tilde{w}_{i})$ and $M_{0} = M|_{t=0}$.

With this notation we state the second main result.

\begin{main} \label{Thm.MainInv2}
Let the geometrical control condition \ref{Assump.001} hold. If 
\begin{itemize}
\item[(i)] there exists a constant $\delta > 0$ such that $|\det{M_{0}}(x)| \geq \delta$ for a.a. $x \in \Omega$, and
\item[(ii)] $\sigma \in C^{1}([0,\tau]; C(\clo{\Omega}))$ and $\Sigma \in C^{1}([0,\tau]; C(\clo{\Omega}))^{n}$,
\end{itemize}
then the operator 
\begin{align}
\left[ M_{0} + K \right] : H^{0}(\Omega)^{n+1} \to H^{0}(\Omega)^{n+1}
 \label{Eqn:MainOp02}
\end{align}
is Fredholm. Moreover, there exists an open and dense subset $\mathcal{U}$ of $(i) \cap (ii)$ such that for each pair $(\sigma,\Sigma)$ in this set, the operator (\ref{Eqn:MainOp02}) is boundedly invertible. 

Concerning the Inverse Problem \ref{Def.InvProb}, for each pair $(\sigma,\Sigma) \in \mathcal{U}$, the inverse problem is uniquely solvable and the following stability estimate,
\begin{align}
\| c^2 - \tilde{c}^2 \|_{H^{0}(\Omega)} \leq C \,\| \lambda -  \tilde{\lambda} \|_{H^{1}([0,\tau]; H^{0}(\Gamma))},
\end{align}
holds for a positive constant $C$ that remains uniformly bounded for $c,\tilde{c}$ in small bounded sets of $C^{2}(\clo{\Omega})$.
\end{main}

\begin{remark} \label{Rmk:06}
With obvious modifications, Remarks \ref{Rmk:02} and \ref{Rmk:03} made after Theorem \ref{Thm.MainInv} are also valid for Theorem \ref{Thm.MainInv2}.
\end{remark}

We note that in this multi-measurement approach, we are treating the unknowns $f$ and $\nabla f$ as independent from each other. In view of results such as \cite{Ima-Yam-2003,Kli-Yam-2006}, we consider appropriate to attempt to recover $f$ with a single properly chosen initial illumination. In particular, equation (\ref{Eqn.Main03}) can be treated as a first-order partial differential equation with operator-valued coefficients. Formally speaking, the principal part of this equation is given by $(\Sigma_{0} \cdot \nabla)$ and the other terms can be seen as compact perturbations. So one would expect (\ref{Eqn.Main03}) to (generally) have a unique solution if $\Sigma_{0}$ is a vector field satisfying conditions to guarantee a unique global solution of $\Sigma_{0} \cdot \nabla f = h$ for arbitrary $h$. Some of these admissible conditions are reviewed in \cite[Ch. 3]{EvansPDE} and  \cite{Bar-1970}. To be precise, we will make the following assumption.

\begin{assumption} \label{Assump.002}
Recall that $\Gamma$ is the accessible portion of the boundary $\partial \Omega$. Assume that the vector field $\Sigma_{0}$ and the domain $\Omega$ satisfy the following compatibility conditions:
\begin{itemize}
\item[(i)] \textit{Non-characteristic boundary}, that is, $\Sigma_{0}(x) \cdot \nu(x) < 0$ for $x \in \Gamma$ and $\Sigma_{0}(x) \cdot \nu(x) > 0$ for $x \in \partial \Omega \setminus \clo{\Gamma}$  where $\nu$ denotes the outward normal on $\partial \Omega$.
\item[(ii)] \textit{Characteristic flow across the domain}, that is, we assume that the family of characteristic trajectories, given by the solutions of the ODE
\begin{align*}
\dot{x}(s) = \Sigma_{0}(x(s)), \qquad s>0,
\end{align*}
issued from $\Gamma$, covers the domain $\Omega$ and exits through $\partial \Omega \setminus \Gamma$ in finite length and there is $\delta > 0$ such that $|\Sigma_{0}(x)| \geq \delta$ for all $x \in \Omega$.
\end{itemize}
\end{assumption}

These assumptions simply say that the flow of characteristic curves induced by $\Sigma_{0}$ sweeps the domain $\Omega$ in a way that is consistent with the accessible portion of the boundary. 

In preparation to state our third main result, we need to make some more definitions and prove some lemmas. The first item to cover is a Poincar\'{e}--Friedrichs type inequality which can be proven using the method of characteristics \cite[Ch. 3]{EvansPDE} and elementary calculus.
\begin{lemma} \label{Lemma.010}
Let Assumption \ref{Assump.002} hold. If $v \in C^{\infty}(\clo{\Omega})$ such that $v|_{\Gamma} = 0$, then the following estimate
\begin{align*}
\| v \|_{H^{0}(\Omega)} \leq C \| \Sigma_{0} \cdot \nabla v \|_{H^{0}(\Omega)}
\end{align*}
holds for some positive constant $C = C(\Omega,\Gamma,\Sigma_{0})$ independent of $v$.
\end{lemma}

Let $\Hsp$ denote the completion of the set $\{ v \in C^{\infty}(\clo{\Omega}) \, : \, v|_{\Gamma} = 0 \}$ with respect to the norm associated with the following inner product
\begin{align*}
\la u , v \ra_{\Hsp} = \la \Sigma_{0} \cdot \nabla u \, , \, \Sigma_{0} \cdot \nabla v \ra_{H^{0}(\Omega)}.
\end{align*}
From Lemma \ref{Lemma.010}, we see that this is a well-defined inner product which makes $\Hsp$ a Hilbert space. Now we are ready to state our final result.

\begin{main} \label{Thm.MainInv3}
Let Assumptions \ref{Assump.001} and \ref{Assump.002} hold. If $\sigma \in C^{1}([0,\tau]; C(\clo{\Omega}))$ and $\Sigma \in C^{2}([0,\tau]; C^{1}(\clo{\Omega}))^{n}$, then the operator 
\begin{align}
\Big[ \Sigma_{0} \cdot \nabla + \sigma_{0} \Big] + (\nabla \cdot \Pop \dot{\Sigma} \Sop )^{*}  + ( \Pop \dot{\sigma} \Sop )^{*} : \Hsp \to H^{0}(\Omega) \label{Eqn:MainOp03}
\end{align}
is Fredholm. Moreover, there exists an open and dense subset $\mathcal{U} \subset C^{1}([0,\tau]; C(\clo{\Omega})) \times C^{2}([0,\tau]; C^{1}(\clo{\Omega}))^{n}$ (satisfying Assumption \ref{Assump.002}) such that for each pair $(\sigma,\Sigma)$ in this set, the operator (\ref{Eqn:MainOp03}) is boundedly invertible. 

Concerning the Inverse Problem \ref{Def.InvProb}, provided that $c = \tilde{c}$ on $\Gamma$, then for each pair $(\sigma,\Sigma) \in \mathcal{U}$, the inverse problem is uniquely solvable and the following locally Lipschitz stability estimate,
\begin{align}
\| c^2 - \tilde{c}^2 \|_{\Hsp} \leq C \, \| \lambda -  \tilde{\lambda} \|_{H^{1}([0,\tau]; H^{0}(\Gamma))},
\end{align}
holds for a positive constant $C$ that remains uniformly bounded for $c$ and $\tilde{c}$ in small bounded sets of $C^{2}(\clo{\Omega})$.
\end{main}

We wish make the following remarks.

\begin{remark} \label{Rmk:30}
Assumption \ref{Assump.002} cannot hold in dimension $n=2$ since we have that $\Sigma \equiv 0$ which follows from (\ref{Eqn:10}). Theorem \ref{Thm.MainInv3} is intended for dimension $n \geq 3$.
\end{remark}

\begin{remark} \label{Rmk:31}
Recall the definitions of $\sigma$ and $\Sigma$ found in (\ref{Eqn:10}). The regularity conditions of Theorem \ref{Thm.MainInv3} require that the initial Cauchy data $(\tilde{w},\partial_{t}\tilde{w})|_{t=0}=(\alpha, \tilde{\beta})$ and the Dirichlet boundary data $\gamma$ to be sufficiently smooth.
We note that if $\tilde{w} \in C^2([0,\tau] \times \clo{\Omega})$ (a classical solution of the wave equation) then this smoothness requirement is satisfied. Also notice that one of the regularity conditions implies that $\Sigma_{0}$ is Lipschitz continuous on $\clo{\Omega}$ which guarantees the existence and uniqueness of the characteristics curves discussed in Assumption \ref{Assump.002}. 
\end{remark}

\begin{remark} \label{Rmk:32}
Aside from sufficient regularity, notice that Theorem \ref{Thm.MainInv3} does not require any a-priori knowledge of the initial velocity $\beta$ in the forward problem (\ref{Eqn:Forward}).
\end{remark}

\begin{remark} \label{Rmk:33}
There is a similitude between (the principal part of) the operator (\ref{Eqn:MainOp03}) and a fundamental first-order equation found in \cite[Eqn. 2.6]{Ima-Yam-2003}, and also in the analysis of stationary inverse problems with internal measurements, such as \cite[Eqn. 7]{Bal-Uhl-2010}, \cite[Eqn. 34]{Chen-Yang-2012}, and \cite[Eqn. 59]{Chen-Yang-2013}. This is due to the fact that the unknown function is found within the principal part of the elliptic operator.
\end{remark}

Before going into the proofs, we would like to provide some examples for initial conditions $\alpha$ in the forward problem (\ref{Eqn:Forward}) so that $\sigma_{0} = \Aop_{g} \alpha$ and $\Sigma_{0} = \frac{2-n}{2} \nabla_{g} \alpha$ satisfy the \textit{positivity} and \textit{flow} conditions needed for Theorems \ref{Thm.MainInv} and \ref{Thm.MainInv3}, respectively. 

\begin{example} \label{Example:01}
In Theorem \ref{Thm.MainInv}, it is required that $\sigma_{0} \geq \delta > 0$. Then, we can simply choose $\alpha$ to be the solution of the Poisson equation $\Aop_{g} \alpha = \delta$ with vanishing Dirichlet values on $\partial \Omega$ where $\delta = \text{const.} > 0$. This solution $\alpha$ is guaranteed to exist for any Riemannian metric $g$ since $\Aop_{g}$ is coercive on $H^{1}_{0}(\Omega)$.
\end{example}

\begin{example} \label{Example:02}
In Theorem \ref{Thm.MainInv3}, it is required that $\Sigma_{0}$ satisfies the flow conditions of Assumption \ref{Assump.002}. This is intended for dimension $n \geq 3$. Let $\alpha$ be chosen as the solution of the following Laplace boundary value problem,
\begin{align*}
&  \mathcal{A}_{g} \alpha = 0 \qquad  && \text{in} \quad \Omega,   \\
& \alpha = -1 \qquad  && \text{on} \quad \Gamma, \\
& \alpha = 0 \qquad  && \text{on} \quad (\partial \Omega \setminus \Gamma). 
\end{align*}
Then, by the \textit{maximum principle} or Hopf lemma for elliptic equations (see for instance \cite[Lemma 3.4]{Gil-Tru-1998}), we have that $\partial_{\nu,g} \alpha(x) < 0$ for $x \in \Gamma$ and $\partial_{\nu,g} \alpha(x) > 0$ for $x \in \partial \Omega \setminus \clo{\Gamma}$ as required by condition $(i)$ of Assumption \ref{Assump.002}. Next we consider the contour surfaces,
\begin{align*}
\Gamma_{s} = \{ x \in \Omega \, : \, \alpha(x) = s \} \qquad s \in [-1,0].
\end{align*}
As long as these contour surfaces are smooth and map $\Gamma$ onto $\partial \Omega \setminus \clo{\Gamma}$ homotopically, then the Hopf lemma applies at each contour surface implying that $ \nabla_{g} \alpha$ cannot vanish in $\Omega$. This renders the positivity condition in part $(ii)$ of Assumption \ref{Assump.002}. For sake of simplicity, we have ignored the regularity conditions in this example. Specifically, there is a singularity at the interface between $\Gamma$ and $(\partial \Omega \setminus \Gamma)$ where the Dirichlet boundary condition for $\alpha$ has a jump discontinuity. A mollification argument could be employed to overcome this limitation, but we do not purse that route any further.
\end{example}

\section{Proof of Main Results} \label{Section:Proof}
In this section we proceed to prove the main results of our paper.

\begin{proof}[Proof of Main Theorem \ref{Thm.MainInv}]
Notice that the first term of the governing operator in (\ref{Eqn:MainOp01})  is boundedly invertible on $H^{0}(\Omega)$ provided that $|\sigma_{0}(x)| \geq \delta > 0$ for a.a. $x \in \Omega$. The second term of the operator is a compact operator on $H^{0}(\Omega)$ as asserted by Lemma \ref{Lemma.005} below. Hence, we obtain a Fredholm equation. 

Now we prove the existence of an open and dense subset $(i) \cap (ii)$ on which (\ref{Eqn:MainOp01}) is boundedly invertible. First, standard perturbation shows that the set of $\sigma$'s over which (\ref{Eqn:MainOp01}) is invertible in $H^{0}(\Omega)$ is open. To show denseness, consider replacing $\sigma$ with
\begin{eqnarray*}
\rho(\lambda) = \lambda \sigma + (1-\lambda) \sigma_{0}, \qquad \text{with $\lambda \in \mathbb{C}$}.
\end{eqnarray*}
Notice now that the first term in (\ref{Eqn:MainOp01}) remains unchanged for any choice of $\lambda \in \mathbb{C}$, and the second term remains compact and analytic with respect to $\lambda \in \mathbb{C}$ because $\dot{\rho}(\lambda) = \lambda \dot{\sigma}$ for all $\lambda \in \mathbb{C}$. If we set $\lambda = 0$, then the operator in (\ref{Eqn:MainOp01}) simply becomes $\sigma_{0}$ which is boundedly invertible provided that $(i)$ holds. By the analytic Fredholm theorem \cite{Ren-Rog-2004}, then the system is boundedly invertible for all but a discrete set of $\lambda$'s. In particular, this holds for values arbitrarily close to $\lambda=1$. This shows the desired denseness. 

The local uniformity of the constant $C$ in Theorem \ref{Thm.MainInv} does not follow directly from the arguments so far. So we proceed to prove this claim for a $\sigma \in \mathcal{U}$. Notice that the operator (\ref{Eqn:MainOp01}) and its inverse depend on $\sigma$ and $\Sop$ which in turn depend on $\tilde{c}$ and $c$, respectively.

First, we address perturbations of $\sigma$ within the space $C^{1}([0,\tau];C(\clo{\Omega}))$. From the definition in (\ref{Eqn:10}), we see that $\sigma$ depends on the evolution operator (c$_0$-semigroup) for the wave equation with speed $\tilde{c}$. This is a problem of regular perturbation in the framework of classical solutions. In the appropriate norms, the evolution operator for a well-posed hyperbolic PDE is locally Lipschitz continuous with respect to the coefficients of the PDE. This follows from perturbation theory of semigroups in Banach spaces. See \cite[Ch. 3]{Eng-Nag-2000} and \cite[Ch. 9]{Kato-1995}. In particular, we have local Lipschitz stability of $\sigma \in C^{1}([0,\tau] ; C(\clo{\Omega}))$ with respect to $\tilde{c} \in C^{2}(\clo{\Omega})$. In fact, using the stability in classical H\"{o}lder spaces \cite{Gil-Tru-1998} of elliptic operators (semigroup generators), we obtain that $\tilde{c} \in C^{1,s}(\Omega)$ for $s > 0$ suffices for our purposes. 

Similarly, the operator $\Sop : H^{0}(\Omega) \to H^{0}((0,\tau) \times \Omega)$ is locally Lipschitz stable with respect to $c \in C^{2}(\clo{\Omega})$. This follows from the manner (HUM method \cite{Glow-Lions-He-2008,Bar-Leb-Rau-1992,Lions-1988}) in which the control operator $\Cop : H^{0}(\Omega) \to H^{0}((0,\tau) \times \Gamma)$ is constructed. The operator $\Cop$ is given by the (bounded) inverse of a certain composition of evolution operators (semigroups). In turn, these evolution operators are locally Lipschitz stable with respect to $c \in C^{2}(\clo{\Omega})$. In fact, in this generalized framework, perturbation of $c \in L^{\infty}(\Omega)$ suffices \cite{Blazek-2013}. Hence, by well-known perturbation arguments \cite{Kato-1995}, we obtain both the locally Lipschitz continuity of $\Cop$ and the invariance of the control time $\tau$ under small perturbations of the speed $c \in C^{2}(\clo{\Omega})$. 

Finally, we have the operator (\ref{Eqn:MainOp01}) being locally Lipschitz continuous with respect to perturbations of $c$ and $\tilde{c}$ in $C^{2}(\clo{\Omega})$. Therefore the same is true for its inverse \cite{Kato-1995}. This provides the local uniformity of the stability constant $C$ which concludes the proof.  
\end{proof}

\begin{proof}[Proof of Main Theorem \ref{Thm.MainInv2}]
The proof of this theorem is completely analogous to the proof of Theorem \ref{Thm.MainInv}. Every statement holds if we make the obvious modifications to handle the operator-valued matrices in (\ref{Eqn:MainOp02}).
\end{proof}

The proof of Theorem \ref{Thm.MainInv3} is a little more involved. The first item to cover is the well-posedness of the principal part of the operator (\ref{Eqn:MainOp03}). We establish this as a lemma.

\begin{lemma} \label{Lemma.003}
Let Assumption \ref{Assump.002} hold. Then the operator 
\begin{align*}
(\Sigma_{0} \cdot \nabla + \sigma_{0})  : \Hsp \to H^{0}(\Omega) 
\end{align*}
is boundedly invertible.
\end{lemma}

\begin{proof}
The operator is clearly bounded. Since Assumption \ref{Assump.002} holds, by the method of characteristics we obtain that the operator is surjective. Again using the form of the solutions from the method of characteristics, we find that for all $v \in \Hsp$
\begin{align*}
\| v \|_{\Hsp} = \| \Sigma_{0} \cdot \nabla v \|_{H^{0}(\Omega)} \leq \Big( 1 + C \lambda e^{C \lambda} \Big) \| \Sigma_{0} \cdot \nabla v + \sigma_{0} v \|_{H^{0}(\Omega)}
\end{align*}
where $\lambda = \| \sigma_{0} \|_{L^{\infty}(\Omega)}$ and $C = C(\Omega,\Gamma,\Sigma_{0})$ is the constant appearing in Lemma \ref{Lemma.010}. This implies the operator is injective and bounded from below. Therefore its inverse is also bounded. 
\end{proof}

With this lemma we can proceed to the proof of the third main theorem of this paper.
\begin{proof}[Proof of Main Theorem \ref{Thm.MainInv3}]
Notice that the operator in (\ref{Eqn:MainOp03}) can be expressed as
\begin{align*}
\Big[ \Sigma_{0} \cdot \nabla + \sigma_{0} \Big] + (\nabla \cdot \Pop \dot{\Sigma} \Sop )^{*}   + ( \Pop \dot{\sigma} \Sop )^{*} : \Hsp \to H^{0}(\Omega),
\end{align*}
where the operator in square brackets is boundedly invertible as shown in Lemma \ref{Lemma.003}. The operators $( \Pop \dot{\sigma} \Sop )^{*} : H^{0}(\Omega) \to H^{0}(\Omega)$ and $(\nabla \cdot \Pop \dot{\Sigma} \Sop )^{*} : H^{0}(\Omega) \to H^{0}(\Omega)$ are compact as shown in Lemma \ref{Lemma.005}, therefore they are also compact from $\Hsp$ to $H^{0}(\Omega)$.
Applying the inverse of $(\Sigma_{0} \cdot \nabla + \sigma_{0})$ we obtain that 
\begin{align*}
I + \Big[ \Sigma_{0} \cdot \nabla + \sigma_{0} \Big]^{-1} \Big[ (\nabla \cdot \Pop \dot{\Sigma} \Sop )^{*}  + ( \Pop \dot{\sigma} \Sop )^{*} \Big] : \Hsp \to \Hsp,
\end{align*}
is Fredholm. The same argument employed in the proof of Theorem \ref{Thm.MainInv}, based on the analytic Fredholm theory, shows that the set of $\sigma \in C^{1}([0,\tau]; C(\clo{\Omega}))$ and $\Sigma \in C^{2}([0,\tau] ; C^{1}(\clo{\Omega}))^{n}$ for which the above operator is boundedly invertible is both open and dense. The statement concerning the local uniformity of the stability constant $C$ follows from the same argument used in the proof of Theorem \ref{Thm.MainInv}. This concludes the proof.
\end{proof}

Now we proceed to prove a lemma already employed in the proofs of the Main Theorems \ref{Thm.MainInv}-\ref{Thm.MainInv3}.

\begin{lemma} \label{Lemma.005}
If $\sigma \in C^{1}([0,\tau] ; C(\clo{\Omega}))$ then $(\Pop \dot{\sigma} \Sop) : H^{0}(\Omega) \to H^{0}(\Omega)$ is compact. Similarly, if $\Sigma \in C^{2}([0,\tau]; C^{1}(\clo{\Omega}))^{n}$ then $(\Pop \dot{\Sigma} \Sop) : H^{0}(\Omega) \to H^{1}(\Omega)^{n}$ is compact.
\end{lemma}

\begin{proof}
First, we consider $\sigma \in C^{\infty}([0,\tau] \times \clo{\Omega})$, keeping in mind that this condition will be relaxed later. Let $\phi \in H^{0}(\Omega)$ and $\eta = \Cop \phi$ and $\psi = \Sop \phi$. 

Now, let $\varrho = \dot{\sigma} \psi$ which satisfies the following problem,
\begin{align*}
& \partial_{t}^{2} \varrho - \mathcal{A}_{c^{-2}g} \varrho = F \quad && \text{in $(0,\tau) \times \Omega$},  \\
& \varrho = 0 \quad \text{and} \quad \partial_{t} \varrho = 0 \quad && \text{on $\{ t = \tau \} \times \Omega$},  \\
& \varrho = G \quad && \text{on $(0,\tau) \times \Gamma$}, \\ 
& \varrho = 0 \quad && \text{on $(0,\tau) \times (\partial \Omega \setminus \Gamma)$},
\end{align*}
where $F \in C([0,\tau];H^{-1}(\Omega))$ is given by
\begin{align*}
F & = \partial_{t}^{2}(\dot{\sigma} \psi) - \mathcal{A}_{c^{-2}g} (\dot{\sigma} \psi) =  ( \partial^{2}_{t}\dot{\sigma} - \mathcal{A}_{c^{-2}g} \dot{\sigma} ) \psi + 2 ( \ddot{\sigma}  \dot{\psi} - \nabla \dot{\sigma} \cdot \nabla_{c^{-2} g} \psi) 
\end{align*}
and $ G \in H^{0}((0,\tau) \times \Gamma)$ is given by
\begin{align*}
G = \dot{\sigma} \eta.
\end{align*}

Hence, the time-integral $\varphi = \Pop \varrho$ satisfies a stationary problem of the form
\begin{align*}
- \mathcal{A}_{c^{-2}g} \varphi & = \ddot{\sigma}|_{t=0} \phi + \Pop F \quad && \text{in $\Omega$},  \\
\varphi & = \Pop G \quad && \text{on $(0,\tau) \times \Gamma$}, \\ 
\varphi & = 0 \quad && \text{on $(0,\tau) \times (\partial \Omega \setminus \Gamma)$}.
\end{align*}
Using integration by parts in time, we obtain
\begin{align*}
\Pop F & = \Pop \big[ ( \partial^{2}_{t}\dot{\sigma} - \mathcal{A}_{c^{-2}g} \dot{\sigma} ) \psi \big]  + 2 \Pop ( \ddot{\sigma}  \dot{\psi} - \nabla \dot{\sigma} \cdot \nabla_{c^{-2} g} \psi) \\
& = \Pop \big[ ( \partial^{2}_{t}\dot{\sigma} - \mathcal{A}_{c^{-2}g} \dot{\sigma} ) \psi \big] - 2 \ddot{\sigma}|_{t=0} \phi -  2 \Pop ( \dddot{\sigma} \psi   ) - 2 \Pop ( \nabla \ddot{\sigma} \cdot \nabla_{c^{-2} g} \Top \psi )
\end{align*}
and 
\begin{align*}
\Pop G = \Pop ( \dot{\sigma} \eta ) = - \Pop( \ddot{\sigma} \Top \eta ).
\end{align*}

In the above steady problem there is an interior forcing term $\ddot{\sigma}|_{t=0} \phi + \Pop F$ and a boundary forcing term $\Pop G$. By virtue of linearity, we can decompose the solution as $\varphi = \varphi_{1} + \varphi_{2}$ where the components solve the following stationary problems,
\begin{subequations}
\label{Eqn:System01}
\begin{align}
- \mathcal{A}_{c^{-2}g} \varphi_{1} & = \ddot{\sigma}|_{t=0} \phi + \Pop F \quad && \text{in $\Omega$},  \\
\varphi_{1} & = 0 \quad && \text{on $(0,\tau) \times \Gamma$}, \\ 
\varphi_{1} & = 0 \quad && \text{on $(0,\tau) \times (\partial \Omega \setminus \Gamma)$},
\end{align}
\end{subequations}
and
\begin{subequations}
\label{Eqn:System02}
\begin{align}
- \mathcal{A}_{c^{-2}g} \varphi_{2} & = 0 \qquad \qquad \qquad && \text{in $\Omega$},  \\
\varphi_{2} & = \Pop G \qquad && \text{on $(0,\tau) \times \Gamma$}, \\ 
\varphi_{2} & = 0 \qquad && \text{on $(0,\tau) \times (\partial \Omega \setminus \Gamma)$},
\end{align}
\end{subequations}

Now, from the stability estimates (\ref{Eqn:Stab02})-(\ref{Eqn:Stab03}) of Lemma \ref{Lemma.007} and Theorem \ref{Thm.Control}, we find that these forcing terms satisfy 
\begin{align*}
 \| \Pop F \|_{H^{0}(\Omega)} &\lesssim \| \psi \|_{C([0,\tau]; H^{0}(\Omega))}  + \| \nabla_{c^{-2} g} \Top \psi \|_{C([0,\tau]; H^{0}(\Omega))} \lesssim \| \phi \|_{H^{0}(\Omega)},
\end{align*}
and
\begin{align*} 
 \| \Pop G \|_{H^{1}_{0}(\Gamma)} \lesssim  \| \Top \eta \|_{H^{1}_{0}((0,\tau) \times \Gamma)} \lesssim \| \phi \|_{H^{0}(\Omega)}.
\end{align*}

On the other hand, since $\mathcal{A}_{c^{-2}g}$ is coercive, then the above stationary problems (\ref{Eqn:System01})-(\ref{Eqn:System02}) are well-posed boundary value problem in a family of Sobolev scales. In particular, for the solution of problem (\ref{Eqn:System01}) we have that 
\begin{align}
\| \varphi_{1} \|_{H^{2}(\Omega)} \lesssim  \| \phi \|_{H^{0}(\Omega)} + \| \Pop F \|_{H^{0}(\Omega)} \lesssim \| \phi \|_{H^{0}(\Omega)}, \label{Eqn:MainEstim01}
\end{align}
and the solution of problem (\ref{Eqn:System02}) satisfies
\begin{align}
\| \varphi_{2} \|_{H^{3/2}(\Omega)} \lesssim  \| \Pop G \|_{H^{1}_{0}(\Gamma)}  \lesssim \| \phi \|_{H^{0}(\Omega)}.
\label{Eqn:MainEstim02}
\end{align}
Estimates (\ref{Eqn:MainEstim01})-(\ref{Eqn:MainEstim02}) are more easily found in the literature for sufficiently smooth domains \cite{Lio-Mag-Book-1972,McLean2000}, but they happen to be valid in Lipschitz domains as well \cite{Jer-Ken-1995}.

As a result of the estimates (\ref{Eqn:MainEstim01})-(\ref{Eqn:MainEstim02}) and the fact that $\varphi = \varphi_{1} + \varphi_{2}$, we obtain that the mapping $\phi \mapsto \Pop \dot{\sigma} \Sop \phi$ is a bounded operator from $H^{0}(\Omega)$ to $H^{3/2}(\Omega)$. Our claim follows due to the compact embedding of $H^{3/2}(\Omega)$ into $H^{1}(\Omega)$ or into $H^{0}(\Omega)$. 

The above proof is valid for a sufficiently smooth $\sigma$. However, it is straightforward to show that for $\phi \in H^{0}(\Omega)$,
\begin{eqnarray*}
\| \Pop (\dot{\sigma}_{1} - \dot{\sigma}_{2}) \Sop \phi \|_{H^{0}(\Omega)} \lesssim \| \dot{\sigma}_{1} - \dot{\sigma}_{2} \|_{C([0,\tau] \times \clo{\Omega})}  \| \phi \|_{H^{0}(\Omega)}
\end{eqnarray*}
Since the subspace of compact operators is closed, then we conclude that $\Pop \dot{\sigma} \Sop : H^{0}(\Omega) \to H^{0}(\Omega)$ is compact for $\sigma \in C^{1}([0,\tau]; C(\clo{\Omega}))$ which concludes the proof of the first part. 

For the second part, following the exact same proof, we would have that $(\Pop \dot{\Sigma} \Sop) : H^{0}(\Omega) \to H^{0}(\Omega)^{n}$ is compact for $\Sigma \in C^{1}([0,\tau]; C(\clo{\Omega}))^{n}$ which is used to prove Theorem \ref{Thm.MainInv2}. However, for Theorem \ref{Thm.MainInv3} we need $(\Pop \dot{\Sigma} \Sop)$ mapping compactly into $H^{1}(\Omega)^{n}$ so that $(\nabla \cdot \Pop \dot{\Sigma} \Sop) : H^{0}(\Omega) \to H^{0}(\Omega)$ is compact. Hence, we need to include one more derivative in the process.

Following the above proof, it follows that $(\Pop \dot{\Sigma} \Sop) : H^{0}(\Omega) \to H^{1}(\Omega)^{n}$ is compact for sufficiently smooth $\Sigma$ since it maps boundedly into $H^{3/2}(\Omega)^{n}$. But we can show that for $\phi \in H^{0}(\Omega)$,
\begin{eqnarray*}
\| \partial_{x} \Pop \dot{\Sigma} \Sop \phi \|_{H^{0}(\Omega)} &= &  \| \partial_{x} \Pop \ddot{\Sigma} \Top \Sop \phi \|_{H^{0}(\Omega)} \\
&\lesssim & \| \partial_{x} \ddot{\Sigma} \|_{C([0,\tau] \times \clo{\Omega})} \| \Top \Sop \phi \|_{H^{0}((0,\tau) \times \Omega)} \\
&& \qquad + \| \ddot{\Sigma} \|_{C([0,\tau] \times \clo{\Omega})} \| \partial_{x} \Top \Sop \phi \|_{H^{0}((0,\tau) \times \Omega)}.
\end{eqnarray*}
Hence, from Lemma \ref{Lemma.007} and Theorem \ref{Thm.Control}, we obtain that
\begin{eqnarray*}
\| \partial_{x} \Pop \dot{\Sigma} \Sop \phi \|_{H^{0}(\Omega)}
&\lesssim & \Big( \| \partial_{x} \ddot{\Sigma} \|_{C([0,\tau] \times \clo{\Omega})} + \| \ddot{\Sigma} \|_{C([0,\tau] \times \clo{\Omega})} \Big) \| \phi \|_{H^{0}(\Omega)}.
\end{eqnarray*}
Therefore, by linearity we obtain
\begin{eqnarray*}
\| \Pop (\dot{\Sigma}_{1} - \dot{\Sigma}_{2}) \Sop \phi \|_{H^{1}(\Omega)} \lesssim \Big( \| \ddot{\Sigma}_{1} - \ddot{\Sigma}_{2} \|_{C([0,\tau]; C^{1}(\clo{\Omega}))} \Big) \| \phi \|_{H^{0}(\Omega)}.
\end{eqnarray*}
Again, because the subspace of compact operators is closed, we conclude that $\Pop \dot{\Sigma} \Sop : H^{0}(\Omega) \to H^{1}(\Omega)^{n}$ is compact for $\Sigma \in C^{2}([0,\tau]; C^{1}(\clo{\Omega}))^{n}$ which concludes the proof of the second part.
\end{proof}



\section*{Acknowledgments}
This work was partially supported by the AFOSR Grant
FA9550-12-1-0117 and the ONR Grant N00014-12-1-0256.
The author would like to thank the anonymous referees for carefully reviewing the manuscript and for providing helpful recommendations.



\bibliographystyle{aims}
\bibliography{Biblio}

\medskip
Received xxxx 20xx; revised xxxx 20xx.
\medskip

\end{document}